\documentclass{article}
\usepackage[utf8]{inputenc}
\usepackage{csquotes}
\usepackage{amsmath, amsfonts, amssymb, amsthm, bigints,esint, mathtools, mathrsfs}
\newtheorem{theorem}{Theorem}

\newtheorem{lemma}{Lemma}

\newtheorem*{remark}{Remark}

\usepackage{geometry, subcaption, outlines}
\usepackage{graphicx, wrapfig}
\usepackage[export]{adjustbox}
\usepackage{todonotes, comment}

\usepackage[backend=biber, sorting=none, style = nature]{biblatex}
\bibliography{Babel}

\title{A strange term coming from the boundary data.}
\author{Aaron Pim}
\date{}
\begin{document}
\maketitle
\begin{abstract}
In this paper, I derive the limiting behaviour of the solutions to Poisson's equation, in a perforated domain, subject to inhomogeneous Robin boundary conditions. In the first half of the paper, I derive a generalised limit for non-periodic domains and arbitrary boundary data. In the second half of this paper, I demonstrate that for periodically arranged spheres and identical Robin boundary data on each sphere, the homogenised limit of Poisson's equation satisfies the Helmholtz equation with an additional term in the domain data, which represents the contribution from the inhomogeneous Robin boundary data. These results are a generalisation of the work of Kaizu \cite{KAIZUS1985TRpo}, who derived the limit of the solutions to the homogeneous Robin problem.
\end{abstract}
\section{Introduction}
    Recently, the modelling of \textit{metamaterials} has been a major area of interest in material science. The non-classical behaviour, that the materials exhibit because of their microstructure, has major applications in technologies, such as enhanced photovoltaic cells, miniaturised antenna systems and lenses that surpass the diffractions limit \cite{YuPeng2019BMA,NakanoHisamatsu2017Lnam,PendryJ.B2000Nrma}. Media whose non-classical behaviours may be adjusted through externally tunable components are referred to as \textit{tunable} metamaterials. An example of materials that are tunable, are \textit{colloidal-nematic suspensions}, which are a mixture of microscopic, insoluble colloidal particles in a solvent of nematic liquid crystal. The colloidal particles interact with the nematic host through intermolecular forces, this phenomenon is called \textit{weak anchoring} and is typically represented by inhomogeneous Robin boundary conditions. \\
    Much research in the past few years has been dedicated to modelling inhomogeneous materials which behave like a \textit{homogeneous} metamaterial. For example, Cioranescu and Murat proved that the homogenised limit of solutions to Poisson's equation, subject to homogeneous Dirichlet conditions, satisfies the Helmholtz equation \cite{Strange_term}. The corresponding Helmholtz coefficient, denoted $\eta$, is referred to as ``strange term coming from nowhere'' and is defined in terms of the shape and relative size of the inclusions. This ``strange term'' corresponds to the non-classical behaviour observed in metamaterials. Cioranescu and Murat's result was further developed by Kaizu, who considered homogeneous Robin conditions and derived a similar term \cite{KAIZUS1985TRpo}. The main goal of this paper is to understand how the intermolecular forces between the inclusions and the host changes the effective behaviour of a nematic-colloidal suspension. To accomplish this, I shall consider the limit of the solutions of Poisson's equation subject to inhomogeneous Robin boundary conditions. 
\section{The homogenised limit of the Robin Problem}
    \subsection{Assumptions on the regularity of the inclusions and boundary data}
        For $N \in \mathbb{N}, N\geqslant 2$, let $\Omega \subset \mathbb{R}^N$ be a bounded domain, which denotes the domain in the absence of inclusions, with smooth boundary $\partial \Omega$ and unit normal $\mathbf{n}_{\Omega}:\partial \Omega \rightarrow \mathbb{S}^{N-1}$. Let the subset $\Omega_\epsilon \subset \Omega$ denote the domain with inclusions. I shall assume that $\Omega_\epsilon$ is such that the set of inclusions, given by 
        \begin{equation}\nonumber
            D_\epsilon := \Omega \setminus \Omega_\epsilon,    
        \end{equation}
        are such that the closure of $D_\epsilon$ does not intersect the boundary $\partial \Omega$, the interior of $D_\epsilon$ is non-empty with Lipschitz continuous boundary and 
        \begin{equation}\label{chapter 6: convergence of the characterstic function}
            \mathbf{1}_{\Omega_\epsilon}\rightarrow 1 \text{ in } L^2(\Omega) \text{ as }\epsilon \rightarrow 0,
        \end{equation}
        where $\mathbf{1}_{A}:\Omega \rightarrow \{0,1\}$ denotes the characteristic function of the set $A \subset \Omega$. As the boundary of $\Omega_\epsilon$ is Lipschitz continuous, this implies \cite[(Lemma 1.5.1.9)]{Grisvard} that there exists a sequence of constants $\delta_\epsilon >0$ and functions $\mu_\epsilon \in C^1(\overline{\Omega}_\epsilon)$ such that 
        \begin{equation}\label{defn of mu and delta}
            (\mu_\epsilon \cdot \mathbf{n})(\mathbf{x}) \geqslant \delta_\epsilon, \quad \forall \mathbf{x}\in \partial \Omega_\epsilon,
        \end{equation}
        where $\mathbf{n}:\partial \Omega_\epsilon \rightarrow \mathbb{S}^{N-1}$ is the outward pointing normal. I will assume that the inclusions $D_\epsilon$ are such that
        \begin{equation}\label{boundary assumption 1}
            \limsup_{\epsilon \rightarrow 0} \dfrac{\|\mu_\epsilon\|_{C^1(\Omega_\epsilon)}}{\delta_\epsilon}<\infty.
        \end{equation}
        The objective of this paper is to understand how inhomogeneous boundary conditions affect the limiting behaviour of the solutions to the Robin problem. Let the sequence $g^\epsilon \in L^2(\partial \Omega_\epsilon)$, $\epsilon>0$, denote the Robin boundary data along $\partial \Omega_\epsilon$, it is assumed that there exists a $g \in L^2(\partial \Omega)$ such that 
        \begin{equation}\label{convergence of the outer boundary data}
            g^\epsilon|_{\partial \Omega}\rightarrow g, \quad \text{ in }L^2(\partial \Omega).
        \end{equation}
        Additionally, let the convergent sequence $\alpha_\epsilon >0$ denote the sequence of Robin coefficients along $\partial D_\epsilon$, let $\alpha >0$ denote the Robin coefficient along $\partial \Omega$. I shall assume that $g_\epsilon$ and $\alpha_\epsilon$, are such that
        \begin{equation}\label{chapter 6: condition on g eps}
            \limsup\limits_{\epsilon \rightarrow 0}{\alpha_\epsilon \|g^\epsilon\|_{L^2(\partial D_\epsilon)}} < \infty
        \end{equation}
        If I were to consider the homogenous Robin problem, then I would be able to express Poisson's equation using a bilinear form. However, such a formulation does not exist because of the inhomogeneity in the boundary conditions, and consequently I consider the asymmetric functional $a^\epsilon:H^1(\Omega_\epsilon)\times H^1(\Omega_\epsilon)$ given by,
        \begin{equation}\nonumber
            a^\epsilon(v_1,v_2):=\int\limits_{\Omega_\epsilon}\nabla v_1 \cdot \nabla v_2 \ d\mathbf{x}+\alpha_\epsilon \int\limits_{\partial D_\epsilon}(v_1-g^\epsilon)v_2 \ dS_{\mathbf{x}}+\alpha \int\limits_{\partial \Omega}(v_1-g^\epsilon)v_2 \ dS_{\mathbf{x}}.
        \end{equation}
        It is clear that if $g^\epsilon =0$ then the $a^\epsilon$ is equal to the bilinear functional $B^\epsilon:H^1(\Omega)\times H^1(\Omega) \rightarrow \mathbb{R}$, given by
        \begin{equation}\nonumber
            B^\epsilon(v_1,v_2):=\int\limits_{\Omega_\epsilon}\nabla v_1 \cdot \nabla v_2 \ d\mathbf{x}+ \alpha_\epsilon \int\limits_{\partial D_\epsilon}v_1v_2 \ dS_{\mathbf{x}}+ \alpha \int\limits_{\partial \Omega}v_1v_2 \ dS_{\mathbf{x}}.
        \end{equation}
        $B^\epsilon$ corresponds to the bilinear formulation for the homogeneous problem, and I assume that the functional is coercive with constant coefficient $c_{\rm coer}>0$,
        \begin{equation}\label{chapter 6: lower bound on bilinear}
            B^\epsilon(v,v)\geqslant c_{\rm coer}\|u\|^2_{H^1(\Omega_\epsilon)}, \quad \forall u \in H^1(\Omega_\epsilon), \quad \forall \epsilon > 0.
        \end{equation}
        To derive the limiting behaviour of the solutions to the Robin problems, I shall extend the solution to Poisson's equation into the inclusions. I assume that there exists a family of extension maps $T^\epsilon \in \mathcal{L}\left( H^1(\Omega_\epsilon), H^1(\Omega)\right)$ such that
        \begin{equation}\label{chapter 6: extension assumption 1}
            \limsup\limits_{\epsilon \rightarrow 0}\|T^\epsilon\|_{\rm op} < \infty,
        \end{equation}
        \begin{equation}\label{chapter 6: bounded op implies converg}
            a^\epsilon(v^\epsilon, v^\epsilon) < \infty \quad \Rightarrow \quad \lim\limits_{\epsilon \rightarrow 0}\|T^\epsilon v^\epsilon - Z^\epsilon v^\epsilon\|_{L^2(\Omega)} \rightarrow 0.
        \end{equation}
        In the above assumption, the sequence $v^\epsilon \in H^1(\Omega_\epsilon)$, $\epsilon>0$, is such that $Z^\epsilon v^\epsilon$ is convergent in $L^2(\Omega)$, where the family of extension maps, denoted $Z^\epsilon:L^2(\Omega_\epsilon) \rightarrow L^2(\Omega)$, extends a function by zero into the inclusions. Additionally, I will assume that there exists a sequence $q^\epsilon \in W^{1,\infty}(\Omega_\epsilon)$, $\epsilon >0$, that satisfies
        \begin{align}
            q^\epsilon &=1 & \text{ on }\partial \Omega \label{chapter 6: BC of q} \\
            T^\epsilon q^\epsilon  &\rightharpoonup 1 & \text{ in }H^1(\Omega) \text{ as }\epsilon \rightarrow 0 \label{chapter 6: extension of q weak conv}
            \\
            a^\epsilon(q^\epsilon, q^\epsilon)&<\infty & \forall \epsilon >0. \label{chapter 6: a eps and q eps is finite}
        \end{align}
        Additionally, I assume $q^\epsilon$ is such that for all $v \in W^{1,1}(\Omega)$ and all sequences $\{v^\epsilon\}_{\epsilon>0} \subset W^{1,1}(\Omega)$ that weakly converge to $v$ in $W^{1,1}(\Omega)$, the expression
        \begin{equation}\label{chapter 6: eta eps defn}
            \int\limits_{\Omega_\epsilon}\nabla v^\epsilon \cdot \nabla q^\epsilon \ d \mathbf{x} + \alpha^\epsilon \int\limits_{\partial D_\epsilon} v^\epsilon q^\epsilon \ dS_{\mathbf{x}},
        \end{equation}
        converges and the limit is independent of the choice of $v^\epsilon$. I denote this limit by $\eta[v]$, which accounts for the ``strange term'' part in the behaviour of the Laplace operator when considering the homogeneous Robin problem. In the aforementioned paper, Kaizu proved that for $\alpha = \alpha^\epsilon >0$ and 
        \begin{equation}\label{chapter 6: example holes}
            D_\epsilon := \bigcup_{\mathbf{x}\in \mathbb{L}_\epsilon} \overline{B}_{r^\epsilon}(\mathbf{x}), \quad \mathbb{L}_\epsilon := \left\{\left.\mathbf{x}\in 2\epsilon \mathbb{Z}^2 \right| \text{dist}(\mathbf{x}, \partial \Omega) \geqslant \epsilon \right\},
        \end{equation}
        where $S:= \lim\limits_{\epsilon \rightarrow 0} |\mathbb{L}_{\epsilon}|r_\epsilon \in (0,\infty)$, then the functional $\eta$ is given by 
        \begin{equation}\label{kaizu eta}
            \eta[u] := \dfrac{\alpha S \omega_N}{|\Omega|} \int\limits_{\Omega}u \ d\mathbf{x},
        \end{equation}
        where $\omega_N$ is the surface area of the unit $N$-ball. To account for the inhomogeneous boundary conditions, I assume furthermore that $D_\epsilon$, $q^\epsilon$ and $g^\epsilon$ are such that for all $\zeta \in C^\infty(\overline{\Omega})$, the sequence
        \begin{equation}\label{chapter 6: assumption of gamma}
            \alpha^\epsilon \int\limits_{\partial D_\epsilon} \zeta  g^\epsilon q^\epsilon  \ dS_{\mathbf{x}},
        \end{equation}
        converges as $\epsilon \rightarrow 0$. Similar to $\eta$, I denote the limit of the above sequence by $\gamma:C^\infty(\overline{\Omega})\rightarrow \mathbb{R}$. It is clear that when the boundary data is identically zero, then the functional $\gamma$ is also zero, recovering the result which Kaizu derived.
    \subsection{The homogenised limits of the solutions to the inhomogeneous Robin problems in non-periodic domains}
        For a given $f \in L^2(\Omega)$, let the sequence $u^\epsilon \in H^1(\Omega_\epsilon)$, $\epsilon >0$, denote the solutions to the following equation
        \begin{align}
            -\Delta u^\epsilon &= f \text{ on }\Omega_\epsilon & \dfrac{\partial u^\epsilon}{\partial \mathbf{n}} + \alpha^\epsilon u^\epsilon &= \alpha^\epsilon g^\epsilon \text{ on } \partial D_\epsilon& \dfrac{\partial u^\epsilon}{\partial \mathbf{n}} + \alpha u^\epsilon &= \alpha g^\epsilon \text{ on } \partial \Omega. \nonumber
        \end{align}
        In this section, I shall derive the weak limit of the sequence of extensions $\left\{T^\epsilon u^\epsilon\right\}_{\epsilon >0}\subset H^1(\Omega)$ under the assumptions from the previous section. I shall begin by proving that there exists a weakly convergent subsequence of $T^\epsilon u^\epsilon$. Afterwards, the limit of this convergent subsequence, denoted $u \in H^1(\Omega)$, will be derived.
        \begin{lemma}\label{lemma proof of weakly convegent subsequence}
            Let the sequence $u^\epsilon \in H^1(\Omega_\epsilon)$, $\epsilon >0$, satisfy the weak form of Poisson's equation subject to inhomogeneous Robin boundary conditions,
            \begin{equation}\label{chapter 6: defn of u eps}
                a^\epsilon(u^\epsilon,v) = \int\limits_{\Omega_\epsilon}fv \ d\mathbf{x} \quad \forall v \in H^1(\Omega_\epsilon), \epsilon >0.
            \end{equation}
            The sequence $\{T^\epsilon u^\epsilon\}_{\epsilon>0} \subset H^1(\Omega)$ is such that,
            \begin{equation}\label{chapter 6: statement of bounded norm}
    	        \limsup\limits_{\epsilon \rightarrow 0}\|T^\epsilon u^\epsilon\|_{H^1(\Omega)}<\infty.
            \end{equation}
        \end{lemma}
        \begin{proof}
            The assumption from equation (\ref{chapter 6: extension assumption 1}) implies that the limit supremum of the operator norm of $T^\epsilon$ is bounded. Consequently, the following inequality holds
            \begin{equation}\label{bound 1, proof 1}
                \|T^\epsilon u^\epsilon\|_{H^1(\Omega)}\leqslant \|T^\epsilon\|_{\rm op}\|u^\epsilon\|_{H^1(\Omega_\epsilon)}.
            \end{equation}
            Thus to prove the bound in equation (\ref{chapter 6: statement of bounded norm}), I seek an upper bound for the sequence $\|u^\epsilon\|_{H^1(\Omega_\epsilon)}$. The definition of $a^\epsilon$ is such that the functional may be expressed as the sum of the bilinear functional $B^\epsilon$ and a boundary integral. As a consequence of the assumption from equation (\ref{chapter 6: lower bound on bilinear}) and the Cauchy-Schwarz inequality, the following bound is obtained
            \begin{align}
                a^\epsilon(u^\epsilon,u^\epsilon) &\geqslant c_{\rm coer}\|u^\epsilon\|^2_{H^1(\Omega_\epsilon)} - \alpha_\epsilon\|g^\epsilon\|_{L^2(\partial D_\epsilon)}\|u^\epsilon\|_{L^2(\partial D_\epsilon)}- \alpha\|g^\epsilon\|_{L^2(\partial \Omega)}\|u^\epsilon\|_{L^2(\partial \Omega)}, \nonumber \\
                a^\epsilon(u^\epsilon,u^\epsilon) &\geqslant c_{\rm coer}\|u^\epsilon\|^2_{H^1(\Omega_\epsilon)} - \left(\alpha_\epsilon\|g^\epsilon\|_{L^2(\partial D_\epsilon)}+ \alpha\|g^\epsilon\|_{L^2(\partial \Omega)}\right)\|u^\epsilon\|_{L^2(\partial \Omega_\epsilon)}.\label{proof 1 equation 1}
            \end{align}
            I wish to construct an upper bound for $\|u^\epsilon\|_{L^2(\partial \Omega_\epsilon)}$ in terms of $\|u^\epsilon\|_{H^1(\Omega_\epsilon)}$. To achieve this I utilise the following trace theorem \cite[(Theorem 1.5.1.10)]{Grisvard},
            \begin{equation}\label{proof 1 equation 2}
                \|v\|^2_{L^2(\partial \Omega_\epsilon)}\leqslant \dfrac{\|\mu_\epsilon \|_{C^1(\overline{\Omega}_\epsilon)}}{\delta_\epsilon}\left(\phi\|\nabla v\|_{L^2(\Omega_\epsilon)}^2+\left(1+\dfrac{1}{\phi} \right)\|v\|_{L^2(\Omega_\epsilon)}^2 \right),\quad \forall v \in H^1(\Omega_\epsilon), \ \phi \in (0,1),
            \end{equation}
            where $\delta_\epsilon>0$ and $\mu_\epsilon \in C^1(\Omega_\epsilon)$ are defined in equation (\ref{defn of mu and delta}). The function $u^\epsilon$ is the solution to the Poisson problem, and consequently equation (\ref{chapter 6: defn of u eps}); thus, by applying the Cauchy-Schwarz inequality, I may construct an upper bound for $a^\epsilon(u^\epsilon,u^\epsilon)$, given by
            \begin{equation}\nonumber
                a^\epsilon(u^\epsilon,u^\epsilon) \leqslant \|f\|_{L^2(\Omega_\epsilon)}\|u^\epsilon\|_{H^1(\Omega_\epsilon)}.
            \end{equation}
            Thus applying the inequalities from equations (\ref{bound 1, proof 1}), (\ref{proof 1 equation 1}) and (\ref{proof 1 equation 2}), the following bound is obtained
            \begin{equation}\nonumber
                c_{\rm coer}\|T^\epsilon u^\epsilon\|_{H^1(\Omega_\epsilon)} \leqslant \|T^\epsilon\|_{\rm op} \left( \|f\|_{L^2(\Omega_\epsilon)} + \left(\alpha_\epsilon\|g^\epsilon\|_{L^2(\partial D_\epsilon)}+ \alpha\|g^\epsilon\|_{L^2(\partial \Omega)}\right)\sqrt{1+\dfrac{1}{\phi}}\sqrt{\dfrac{\|\mu_\epsilon \|_{C^1(\overline{\Omega}_\epsilon)}}{\delta_\epsilon}}\right).
            \end{equation}
            Thus by the assumptions from equations (\ref{boundary assumption 1}), (\ref{convergence of the outer boundary data}), (\ref{chapter 6: condition on g eps}) and (\ref{chapter 6: extension assumption 1}), it is clear that equation (\ref{chapter 6: statement of bounded norm}) holds.
            \end{proof}
            \begin{remark}
                As a consequence of the above lemma there exists a function $u \in H^1(\Omega)$ and a subsequence $\{T^{\epsilon_n}u^{\epsilon_n}\}_{n=1}^\infty \subset \{T^{\epsilon}u^{\epsilon}\}_{\epsilon >0} \subset H^1(\Omega)$ that satisfies:
                \begin{equation}\nonumber
                	T^{\epsilon_n}u^{\epsilon_n}\rightharpoonup u \text{ in }H^1(\Omega) \text{ as }n \rightarrow \infty.
                \end{equation}
                For tractability in the proof of theorem \ref{chapter 6: theorem proof of main result}, I shall use the following notation:
                \begin{align*}
                    \tilde{u}_n &:= T^{\epsilon_n} u^{\epsilon_n},& u_n &:= u^{\epsilon_n}, &\tilde{q}_n &:= T^{\epsilon_n} q^{\epsilon_n},& q_n &:= q^{\epsilon_n},& a_n &:= a^{\epsilon_n},\\ \Omega_n &:= \Omega_{\epsilon_n},& D_n &:= D_{\epsilon_n},&Z_n&:=Z^{\epsilon_n},&g_n &:= g^{\epsilon_n},&\alpha_n &:=\alpha_{\epsilon_n}.
                \end{align*}
            \end{remark}
        \begin{theorem}\label{chapter 6: theorem proof of main result}
            Let the sequence $u^\epsilon \in H^1(\Omega_\epsilon)$, $\epsilon >0$, satisfy the weak Robin problems, given in equation (\ref{chapter 6: defn of u eps}). The function sequence $\tilde{u}_n$ weakly converges in $H^1(\Omega)$ to the function $u \in H^1(\Omega)$ what satisfies the following integral identity:
            \begin{equation}\label{chapter 6: defn of u}
                0 = \int\limits_{\Omega}\left(\nabla u\cdot \nabla \zeta - f\zeta \right) d \mathbf{x}+\alpha\int\limits_{\partial \Omega}(u-g)\zeta \ dS_{\mathbf{x}} + \eta[u\zeta]-\gamma[\zeta], \qquad \forall \zeta \in C^{\infty}(\overline{\Omega}).
            \end{equation}
        \end{theorem}
        \begin{proof}
            For an arbitrary test function $\zeta \in C^\infty(\overline{\Omega})$, consider the expression $a_n(u_n, q_n \zeta)$. Applying the definition of $u^\epsilon$, from equation (\ref{chapter 6: defn of u eps}), I deduce that
            \begin{equation}\label{chapter 6: un bilin prob}
                \begin{split}
                    0 = \int\limits_{\Omega_n}q_n\left(\nabla u_n \cdot \nabla \zeta - f \zeta \right)\ d\mathbf{x} + \alpha\int\limits_{\partial \Omega}q_n(u_n-g_n)\zeta \ dS_{\mathbf{x}} \\ + \left(\int\limits_{\Omega_n}\zeta \nabla u_n \cdot \nabla q_n \ d\mathbf{x} + \alpha_n\int\limits_{\partial D_n} u_n q_n \zeta \ dS_{\mathbf{x}}\right)-\alpha_n\int\limits_{\partial D_n} g_n q_n \zeta \ dS_{\mathbf{x}}.
                \end{split}
            \end{equation}
            I wish to derive the limit of this expression as $n\rightarrow \infty$.
            \subsubsection*{First term}
            Consider the first term:
            \begin{align}
            	\int\limits_{\Omega_n}q_n\left(\nabla u_n\cdot \nabla \zeta - f\zeta \right) d \mathbf{x} = & \int\limits_{\Omega}Z_nq_n\left(\nabla \tilde{u}_n\cdot \nabla \zeta - f\zeta \right) d \mathbf{x} \nonumber \\
                \int\limits_{\Omega_n}q_n\left(\nabla u_n\cdot \nabla \zeta - f\zeta \right) d \mathbf{x} =  & \int\limits_{\Omega}\tilde{q}_n\left(\nabla \tilde{u}_n\cdot \nabla \zeta - f\zeta \right) d \mathbf{x}  + \int\limits_{\Omega}\left( Z_nq_n -\tilde{q}_n \right)\left(\nabla \tilde{u}_n\cdot \nabla \zeta - f\zeta \right) d \mathbf{x}. \nonumber
            \end{align}
            As $u$ is the weak limit of $\tilde{u}_n$ in $H^1(\Omega)$, this implies that $\nabla \tilde{u}_n\cdot \nabla \zeta$ strongly converges to $\nabla u\cdot \nabla \zeta$ in $L^2(\Omega)$. Thus, as an immediate consequence of equation (\ref{chapter 6: extension of q weak conv}) I have that
            \begin{equation}\nonumber
                \lim\limits_{n \rightarrow \infty} \int\limits_{\Omega}\tilde{q}_n\left(\nabla \tilde{u}_n\cdot \nabla \zeta - f\zeta \right) d \mathbf{x} = \int\limits_{\Omega}\left(\nabla u\cdot \nabla \zeta - f\zeta \right) d \mathbf{x}.
            \end{equation}
            Additionally, as an immediate consequence of equations (\ref{chapter 6: a eps and q eps is finite}) and (\ref{chapter 6: bounded op implies converg}), I have that
            \begin{equation}\nonumber 
                \int\limits_{\Omega}\left( Z_nq_n -\tilde{q}_n \right)\left(\nabla \tilde{u}_n\cdot \nabla \zeta - f\zeta \right) d \mathbf{x} = 0.
            \end{equation}
            Thus I have deduced the limit of the first term
            \begin{equation}\nonumber
            	\lim\limits_{n\rightarrow \infty}\int\limits_{\Omega_n}q_n\left(\nabla u_n \cdot \nabla \zeta - f \zeta \right)\ d\mathbf{x} = \int\limits_{\Omega}\left(\nabla u \cdot \nabla \zeta - f \zeta \right)\ d\mathbf{x}.
            \end{equation}
            \subsubsection*{Second term}
            I shall now consider the second term in equation (\ref{chapter 6: un bilin prob}), as an immediate consequence of the boundary condition of $q_n$, given in equation (\ref{chapter 6: BC of q}), I have that
            \begin{equation}\nonumber
                \alpha \int\limits_{\partial \Omega}(u_n-g_n)q_n\zeta \ dS_{\mathbf{x}} = \alpha \int\limits_{\partial \Omega}(u_n-g_n)\zeta \ dS_{\mathbf{x}}.
            \end{equation}
            The definition as $u_n \rightharpoonup u$ in $H^1(\Omega)$ this implies that $\text{Tr}(u_n) \rightharpoonup \text{Tr}(u)$ in $H^{1/2}(\partial \Omega)$. Additionally, the assumption from equation (\ref{convergence of the outer boundary data}) implies that the limit of the second term is given by 
            \begin{equation}\nonumber
                \alpha\lim\limits_{n \rightarrow \infty }\int\limits_{\partial \Omega}(u_n-g_n)q_n\zeta \ dS_{\mathbf{x}} = \alpha\int\limits_{\partial \Omega}(u-g)\zeta \ dS_{\mathbf{x}}.
            \end{equation}
            \subsubsection*{Third term}
            I shall now consider the third term in equation (\ref{chapter 6: un bilin prob}),
            \begin{equation}\nonumber 
                \int\limits_{\Omega_n}\zeta \nabla u_n \cdot \nabla q_n \ d \mathbf{x} = \int\limits_{\Omega}\mathbf{1}_{\Omega_n}\zeta  \nabla \tilde{u}_n \cdot \nabla \tilde{q}_n \ d \mathbf{x}.
            \end{equation}
            I now use the identity that $\mathbf{1}_{\Omega_n} \zeta \nabla \tilde{u}_n = \mathbf{1}_{\Omega_n}\left(\nabla (\zeta \tilde{u}_n)  - \tilde{u}_n \nabla \zeta \right)$ to deduce that
            \begin{equation}\nonumber
                \int\limits_{\Omega_n}\zeta \nabla u_n \cdot \nabla q_n \ d \mathbf{x} = \int\limits_{\Omega}\mathbf{1}_{\Omega_n}\nabla (\zeta \tilde{u}_n) \cdot \nabla \tilde{q}_n \ d \mathbf{x} - \int\limits_{\Omega}\mathbf{1}_{\Omega_n}\tilde{u}_n \nabla \zeta \cdot \nabla \tilde{q}_n \ d \mathbf{x}.
            \end{equation}
            From equation (\ref{chapter 6: extension of q weak conv}), I have that $\nabla \tilde{q}_n \rightharpoonup \underline{0}$ in $L^2(\Omega,\mathbb{R}^2)$. Additionally, as $u$ is the weak limit of $u^\epsilon$, I may use the assumption from equation (\ref{chapter 6: convergence of the characterstic function}) to deduce that
            \begin{equation}\nonumber
                \mathbf{1}_{\Omega_n}\tilde{u}_n \nabla \zeta \rightarrow u \nabla \zeta \text{ in }L^2(\Omega, \mathbb{R}^N).
            \end{equation}
            Thus I have that
            \begin{equation}\label{chapter 6: bitch 4}
                \lim\limits_{n \rightarrow \infty}\int\limits_{\Omega_n}\zeta \nabla u_n \cdot \nabla q_n \ d \mathbf{x} = \lim\limits_{n \rightarrow \infty}\int\limits_{\Omega}\mathbf{1}_{\Omega_n}\nabla (\zeta \tilde{u}_n) \cdot \nabla \tilde{q}_n \ d \mathbf{x}.
            \end{equation}
            As $\tilde{u}_n \rightharpoonup u$ in $H^1(\Omega)$ it follows that $\zeta \tilde{u}_n \rightharpoonup \zeta u$ in $W^{1,1}(\Omega)$. Thus, I deduce the following limit, using the assumption from equation (\ref{chapter 6: eta eps defn}),
            \begin{equation}\nonumber
                \lim\limits_{n \rightarrow \infty}\left(\int\limits_{\Omega}\mathbf{1}_{\Omega_n}\nabla (\zeta \tilde{u}_n) \cdot \nabla \tilde{q}_n \ d \mathbf{x} + \alpha_n\int\limits_{\partial D_n} \zeta u_n q_n\ dS_{\mathbf{x}} \right) = \eta[\zeta u].
            \end{equation}
            \subsubsection*{Final term}
            Similarly, the assumption from equation (\ref{chapter 6: assumption of gamma}) implies that the fourth term is
            \begin{equation}\nonumber
                \lim\limits_{n\rightarrow \infty}\alpha_n \int\limits_{\partial D_n} q_n g_n\zeta  \ dS_{\mathbf{x}} = \gamma[\zeta] \in (C^\infty)^*(\overline{\Omega}).
            \end{equation}
            Therefore, equation \ref{chapter 6: defn of u} holds, which concludes the proof.
        \end{proof}
\section{The homogenised limit in periodic domains}\label{periodic arragned balls section}
\subsection{Derivation of the auxiliary function}
    In the previous section, I investigated the limiting behaviour of the solutions to inhomogeneous Poisson problems, in non-periodic domains. The derived limit is similar to the limit of the solutions to homogeneous Poisson problems, but contains the additional term $\gamma$, which represents the effect of the boundary data $g^\epsilon$ on the macroscopic behaviour. However, it is not clear, in the non-periodic case, how $\gamma$ affects the solution $u$. Therefore, to demonstrate the difference between the limits of the homogeneous and inhomogeneous problems, I shall consider a set of periodically arranged balls of radius $r_\epsilon < \epsilon$ defined in equation (\ref{chapter 6: example holes}). Let the positions of the centres of the balls be denoted $\mathbb{L}_\epsilon \subset \Omega$, additionally the periodic arrangement of the balls is assumed to satisfy
    \begin{equation}\nonumber
        S_N := \lim\limits_{\epsilon \rightarrow 0}|\mathbb{L}_{\epsilon}|r_\epsilon^{N-1}\in (0,\infty), \quad |\mathbb{L}_\epsilon|\sim \dfrac{|\Omega|}{(2\epsilon)^N}, \quad r_\epsilon \sim \epsilon^{\frac{N}{N-1}}.
    \end{equation}
    Finally, I shall assume that the sequence of Robin parameters $\alpha^\epsilon$ of the inclusion boundaries $\partial D_\epsilon$ is equal to the Robin parameter $\alpha \geqslant 0$ of the boundary $\partial \Omega$. The above assumptions are similar to the ones made by Kaizu, in his derivation of the limit of the solutions of the Homogenous problems. Thus, the ``strange term'' functional $\eta$ is given by equation (\ref{kaizu eta}), the objective of this section is to quantify the limit functional $\gamma$. Thus, I require an explicit expression for the trace of $q^\epsilon$ on $\partial D_\epsilon$; similar to Murat and Cioranescu \cite{Strange_term}, I let $q^\epsilon$ satisfy:
    \begin{align*}
        q_\epsilon &=1 \quad \text{ in }\Omega \setminus B_\epsilon(\mathbb{L}_\epsilon),\\
        -\Delta q_\epsilon &=0 \quad \text{ in }B_\epsilon(\mathbb{L}_\epsilon) \setminus \overline{B}_{r_\epsilon}(\mathbb{L}_\epsilon),\\
        \dfrac{\partial q_\epsilon}{\partial \mathbf{n}}+ \alpha q_\epsilon &=0 \quad \text{ on } \partial B_{r_\epsilon}(\mathbb{L}_\epsilon),
    \end{align*}
    Solving the above equation in the neighbourhood of an arbitrary point in $\mathbb{L}_\epsilon$, it is clear that the function $q_\epsilon$ is given by
    \begin{equation}\nonumber
        q_\epsilon(\mathbf{x}) =\left\{\begin{array}{cl}
            \dfrac{1 + \alpha r_\epsilon \log \left(|\mathbf{x}-\hat{\mathbf{x}}|/r_\epsilon\right)}{1 + \alpha r_\epsilon \log \left(\epsilon/r_\epsilon\right)} & \forall \mathbf{x} \in  B_\epsilon(\hat{\mathbf{x}}) \setminus B_{r_\epsilon}(\hat{\mathbf{x}}) , \ \hat{\mathbf{x}}\in \mathbb{L}_\epsilon, \ N=2 \\
            & \\
            \dfrac{(N-2)+\alpha r_\epsilon \left(1-(r_\epsilon/|\mathbf{x}-\hat{\mathbf{x}}|)^{N-2}\right)}{(N-2)+\alpha r_\epsilon \left(1-(r_\epsilon/\epsilon)^{N-2}\right)} & \forall \mathbf{x} \in  B_\epsilon(\hat{\mathbf{x}}) \setminus B_{r_\epsilon}(\hat{\mathbf{x}}) , \ \hat{\mathbf{x}}\in \mathbb{L}_\epsilon, \ N\geqslant 3 \\
            & \\
            1 & \forall \mathbf{x} \in \Omega \setminus B_\epsilon(\hat{\mathbf{x}}), \ \hat{\mathbf{x}}\in \mathbb{L}_\epsilon, \\
        \end{array}\right.
    \end{equation}
   The definition of the functional $\gamma$ requires the evaluation of $q_\epsilon$ along the boundary of each inclusion. In the case of the ball, $q_\epsilon|_{\partial D_\epsilon}$ is a constant, denoted $q_\epsilon^\partial$, dependent on the radius $r_\epsilon$, the minimum distance $\epsilon$ and the dimension $N$.
    \begin{equation}\label{expansion of q epsi}
        q^{\partial}_\epsilon :=\left\{\begin{array}{cc}
            \left(1+ \alpha r_\epsilon \log \dfrac{\epsilon}{r_\epsilon} \right)^{-1} & N=2 \\
            & \\
            \left(1 + \dfrac{\alpha r_\epsilon}{N-2}  \left(1-\left(\dfrac{r_\epsilon}{\epsilon} \right)^{N-2}\right)\right)^{-1} & N\geqslant 3
        \end{array} \right.
    \end{equation}
    Therefore, for an arbitrary $\zeta \in C^\infty(\overline{\Omega})$, the functional $\gamma[\zeta]$ is defined as the limit of the following
    \begin{equation}\label{sequence of gamma}
        \alpha q_\epsilon^\partial  r_\epsilon^{N-1}\sum\limits_{\hat{\mathbf{x}} \in \mathbb{L}_\epsilon} \ \int\limits_{\mathbb{S}^{N-1}} \left(g^\epsilon \zeta \right)(\hat{\mathbf{x}}+r_\epsilon \mathbf{m}) \ d\mathbf{m}.
    \end{equation}
\subsection{Derivation of the functional $\gamma$}
\subsubsection{Assumption of identical boundary data}
    To quantify the above limit I shall consider the case when the boundary data is identical along the boundary of each inclusion, that there exists a sequence of functions $h^\epsilon \in L^2(\mathbb{S}^{N-1})$ such that 
    \begin{equation}\nonumber
        h^\epsilon(\mathbf{m}) = g^\epsilon(\hat{\mathbf{x}}+r_\epsilon \mathbf{m}), \quad \forall \epsilon >0, \ \hat{\mathbf{x}}\in \mathbb{L}_{\epsilon}, \mathbf{m}\in \mathbb{S}^{N-1}.
    \end{equation}
    As a consequence of the above assumptions, it is clear that the functional $\gamma$ is given by 
    \begin{equation}
        \gamma[\zeta]=\dfrac{\alpha S \omega_N}{|\Omega|}\overline{h}\int\limits_\Omega \zeta \ d\mathbf{x},
    \end{equation}
    where $\overline{h}$ is the limit of the mean of $h^\epsilon$,
    \begin{equation}\nonumber
        \overline{h}:=\lim\limits_{\epsilon \rightarrow 0} \dfrac{1}{\omega_N}\int\limits_{\mathbb{S}^{N-1}}h^\epsilon \ d\mathbf{m}.
    \end{equation}
    As a consequence the limit of the solutions of the inhomogeneous problems, denoted $u \in H^1(\Omega)$, satisfies
    \begin{equation}\label{periodic inclusions inhomogenous limit}
        \left(\dfrac{\alpha \omega_N S}{|\Omega|}-\Delta \right)u = f + \dfrac{\alpha \omega_N S}{|\Omega|}\overline{h} \text{ on }\Omega, \quad \dfrac{\partial u}{\partial \mathbf{n}}+\alpha u = \alpha g \text{ on }\partial \Omega.
    \end{equation}
    Thus the introduction of the boundary data $g^\epsilon$ results in a non-trivial contribution to the domain data $f$.
\subsubsection{Comparison with the limit of the homogenous problem}
    To compare this problem to the result derived by Kaizu, I consider the case when the boundary data is identically given by the constant $g^\epsilon = c_g \in \mathbb{R}$ on $\partial \Omega_\epsilon$. Applying the above result, it is clear that in this case the function $u \in H^1(\Omega)$ would satisfy equation (\ref{periodic inclusions inhomogenous limit}) with $\overline{h}=c_g$. However, one can also consider the limit of the function $v^\epsilon:=u^\epsilon-c_g$; that satisfies the homogeneous problem. As $v^\epsilon \in H^1(\Omega)$ satisfies the homogeneous problem, the result derived by Kaizu implies that the limit $v^\epsilon \rightharpoonup v$, satisfies 
    \begin{equation}\nonumber
        \left(\dfrac{\alpha \omega_N S}{|\Omega|}-\Delta \right)v = f \text{ on }\Omega, \quad \dfrac{\partial v}{\partial \mathbf{n}}+\alpha v = 0 \text{ on }\partial \Omega.
    \end{equation}
    If one substitutes $\tilde{u}:= v+c_g$ it is cleat that $\tilde{u}\in H^1(\Omega)$ satisfies the same equation as $u$, and therefore they are identical.
\section{Conclusion and Future Work}
    In Lemma \ref{lemma proof of weakly convegent subsequence}, I demonstrate that there exists a weakly convergent subsequence of the solutions to the inhomogeneous Robin problem. In Theorem \ref{chapter 6: theorem proof of main result}, I prove that the limit of the weakly convergent subsequence satisfies equation (\ref{chapter 6: defn of u}), where the functional $\eta$ is the ``strange term'' derived by Kaizu, and the functional $\gamma$ is defined as the limit of the sequence of functional's given in equation (\ref{chapter 6: assumption of gamma}). The functional $\gamma$ represents the contribution of the inhomogeneous boundary data in the limit. In Section \ref{periodic arragned balls section}, I consider a sequence of periodically arranged circular inclusions and a sequence of Robin boundary data, denoted $h^\epsilon$, which is identical on each inclusion. These conditions are identical to the ones considered by Kaizu, and I demonstrate that the functional $\gamma$ is a constant contribution to the domain data $f$ that is dependent on the limit of the mean of $h^\epsilon$. I then demonstrated that for the specific subcase of constant boundary conditions, the result matches Kaizu's. \\
    The ability to control the domain data of the effective behaviour of a metamaterial is a significant result which can be expanded upon, with the aim to introduce a controllable non-constant contribution to the domain data $h$. I believe that this may be achieved by considering non-periodically arranged balls and applying the Radon-Nikod{\'y}m Theorem to deduce the limit of the non-uniform distribution of the inclusion centre's. This problem may correspond to clustering of the inclusions, which occurs in nematic-colloidal suspensions \cite{MondalSourav2018Nfcs}. \\ Another avenue of future research would be to consider the homogenised limits of the inhomogeneous Dirichlet and Neumann problems. The limit of the homogeneous solutions of these problems have been considered by the likes of Cioranescu and Murat; but, similar to the Robin problem, the inhomogeneous boundary data may introduce novel effects. \\
    The solution of the limiting problem $u$, is the weak limit of a convergent subsequence. A point of further research would be to strengthen this result and to potentially derive a rate of convergence. This research would be based off of the work of Cherednichenko, Dondl and Rösler \cite{DondlPatrick2017NCiP}, who prove that the resolvent of the Laplace operator $-\Delta|_{\Omega_\epsilon}$ converges, in the sense of the operator norm, to $\eta - \Delta$ for periodically arranged inclusions and homogeneous Dirichlet, Neumann and Robin boundary conditions.
\printbibliography
\end{document}